\documentclass[reqno,11pt]{amsart}

\usepackage{amsmath,amssymb,amsthm,tikz,mathtools}
\usepackage{verbatim}
\usepackage{xcolor}
\usepackage{mathrsfs}

\newtheorem{theorem}{Theorem}[section]

\newtheorem{definition}[theorem]{Definition}
\newtheorem{lemma}[theorem]{Lemma}
\newtheorem{corollary}[theorem]{Corollary}
\newtheorem{proposition}[theorem]{Proposition}

\newtheorem{remark}[theorem]{Remark}
\newtheorem{assumption}{Assumption}[section]

\numberwithin{equation}{section} 

\usepackage{setspace}
\newcommand{\intav}[1]{\mathchoice {\mathop{\vrule width 6pt height 3 pt depth  -2.5pt
\kern -8pt \intop}\nolimits_{\kern -6pt#1}} {\mathop{\vrule width
5pt height 3  pt depth -2.6pt \kern -6pt \intop}\nolimits_{#1}}
{\mathop{\vrule width 5pt height 3 pt depth -2.6pt \kern -6pt
\intop}\nolimits_{#1}} {\mathop{\vrule width 5pt height 3 pt depth
-2.6pt \kern -6pt \intop}\nolimits_{#1}}}

\numberwithin{equation}{section}

\usepackage[usenames,dvipsnames]{pstricks}
\usepackage{epsfig}
\usepackage{pst-grad}
\usepackage{pst-plot}
\usepackage[space]{grffile}
\usepackage{etoolbox}
\makeatletter

\begin{document}

\author{Dami\~ao J. Ara\'ujo}
\address{Departamento de Matem\'atica, Universidade Federal da Para\'iba 58059-900, Jo\~ao Pessoa - Para\'iba (Brazil)}
\email{araujo@mat.ufpb.br}
\author{Andreas Minne}
\address{}
\author{Edgard A. Pimentel}
\address{CMUC, Department of Mathematics, University of Coimbra, 3000-143 Coimbra, Portugal}
\email{edgard.pimentel@mat.uc.pt}

\title[Optimal boundary regularity for free boundary problems]{Fully nonlinear free boundary problems: optimal boundary regularity beyond convexity}
\date{\today}

\scriptsize 

\normalsize

\begin{abstract}
We study a general class of elliptic free boundary problems equipped with a Dirichlet boundary condition. Our primary result establishes an optimal $C^{1,1}$-regularity estimate for $L^p$-strong solutions at points where the free and fixed boundaries intersect. A key novelty is that no convexity or concavity assumptions are imposed on the fully nonlinear operator governing the system. 

Our analysis derives BMO estimates in a universal neighbourhood of the fixed boundary. It relies solely on a differentiability assumption. Once those estimates are available, applying by now standard methods yields the optimal regularity.
\end{abstract}


\maketitle


\section{Introduction}\label{sec_intro}
We consider $L^n$-strong solutions to the unconstrained free boundary problem
\begin{equation}\label{maineq}
    \begin{cases}
        F(D^2u)=1&\hspace{.2in}\mbox{in}\hspace{.2in} B_1^+ \cap \Omega \\
        |D^2 u| \leq K &\hspace{.2in}\mbox{in}\hspace{.2in}B_1^+ \setminus \Omega \\
        u=g & \hspace{.2in}\mbox{in}\hspace{.2in} B_1', \\
\end{cases}
\end{equation}
where $K>0$ is a given constant, $F$ denotes a fully nonlinear elliptic operator, and $g\in C^{2,\alpha}(\overline B_1^+)$, $\alpha \in (0,1)$, is a given boundary data. Here, $B_1^+$ stands for the upper hemisphere of the unit ball $B_1\subset\mathbb{R}^n$, whereas $B_1'$ is the intersection of $\overline B_1$ with the set $\left\lbrace x_n=0\right\rbrace$. The unknown consists of a pair \((u, \Omega)\), where \(u \in W^{2,n}(B_1^+)\) satisfies \eqref{maineq} almost everywhere, and \(\Omega \subset \mathbb{R}^n\) is such that \(\Omega \cap B_1^+\) represents the region where the Hessian of \(u\) is essentially bounded by \(K\). We refer to $\partial\Omega$ as the free boundary. 

We prove boundary $C^{1,1}$-regularity estimates for $L^n$-strong solutions to \eqref{maineq}. Our main contribution is to work under no convexity/concavity assumptions on the operator $F$. Instead, we only require it to be uniformly elliptic and differentiable.

The model \eqref{maineq} amounts to the boundary value problem associated with a broad class of free boundary problems introduced in the work of Figalli and Shahgholian \cite{FS14}; see also \cite{FSoverview,FS15}. In \cite{FS14}, the authors consider $L^n$-strong solutions to 
\begin{equation}\label{inteq}
    \begin{cases}
        F(D^2u)=1 & \hspace{.2in}\mbox{in}\hspace{.2in} B_1 \cap \Omega \\
        |D^2 u| \leq K &\hspace{.2in}\mbox{in}\hspace{.2in} B_1 \setminus \Omega.
    \end{cases}
\end{equation}
Under a convexity/concavity assumption on the uniformly elliptic operator $F$, they prove interior $C^{1,1}$-regularity of the solutions. Once the (optimal) regularity of the solutions is understood, the authors examine the geometry of the free boundary. They prove $\partial\Omega\cap B_r(0)$ to be the graph of a $C^1$-regular function. 

The formulation in \eqref{inteq} accommodates a variety of obstacle-like types of problems. As noted in \cite{FS14}, for $\Omega=\left\lbrace u\neq 0 \right\rbrace$, \eqref{inteq} becomes the fully nonlinear obstacle problem
\[ 
    F(D^2u)=\chi_{\left\lbrace u\neq 0 \right\rbrace}.
 \]
Moreover, for $\Omega=\left\lbrace |Du|\neq 0 \right\rbrace$, see \cite{CS}, then \eqref{inteq} turns into
\[
    F(D^2u)=\chi_{\left\lbrace |Du|\neq 0 \right\rbrace}.
\]

Regarding the regularity of the solutions to \eqref{inteq}, the arguments in \cite{FS14} unravel in two main steps. The authors first notice that an $L^n$-strong solution to \eqref{inteq} is an $L^n$-viscosity solution to an equation of the form $F(D^2u)=f$ in $B_1$, where $f\in L^\infty(B_1)$. Under the convexity of the operator, this observation unlocks $BMO$-estimates for the hessian of $u$. These estimates allow the authors to produce a fully nonlinear counterpart to the arguments in \cite{ALS}. Indeed, instead of considering projections on second-order harmonic polynomials, they work under a projection on polynomials constrained by the BMO estimates stemming from the equation. Here, the dichotomy relates the size of such polynomials and a decay rate for scaled sets depending on $\Omega$. The parabolic counterpart of the findings in \cite{FS14} are reported in \cite{FS15}. An extension of this corpus of results to the context of operators with variable coefficients and explicit dependence on lower-order terms is the subject of \cite{IM2}. In \cite{IM}, the authors study the free and the fixed boundary intersection and prove a non-transversality result in the planar case. Regarding the non-transversal intersection of the free and the fixed boundary in arbitrary dimension, we mention the developments in \cite{Indrei_2019}. See also \cite{Indrei_2019b,Indrei_2024}.

Concerning the boundary regularity of the solutions to \eqref{maineq}, we notice that convexity is critical even if one reduces the problem to an interior one. Indeed, even in the case $g\equiv 0$, by considering an odd reflection of $u$ around the $x_n$-axis, the use of an interior $C^{1,1}$-regularity (e.g., \cite[Theorem 1.2]{FS14}), would still require the convexity of $F$. 

We examine the regularity of $L^n$-strong solutions to \eqref{maineq} in the \emph{absence} of convexity for the operator $F$. Instead, we suppose $F$ is differentiable. Under this condition, we relate strong and viscosity solutions to that problem. More precisely, we recall that $L^p$-viscosity solutions to 
\[
    F(D^2u)=0\hspace{.2in}\mbox{in}\hspace{.2in}B_1^+,
\]
with $u=0$ on $B_1'$, are $C^{2,\alpha}$-regular with estimates in a uniform neighbourhood of the flat boundary, provided $F$ is differentiable \cite[Theorem 1.3]{SS}. In this setting, we prove a BMO-estimate for the $L^n$-strong solutions to \eqref{maineq}. Once such control in average is available, we construct a sequence of polynomials and explore a dichotomy. If this sequence is bounded, its relation with the Hessian of the solutions yields $C^{1,1}$-estimates. Conversely, if this sequence is unbounded, we derive a geometric decay rate for the measure of scaled sets related to the region where the Hessian of $u$ is bounded by $K$. Such a geometric decay frames an auxiliary problem under a pointwise boundary-variant of Caffarelli's $C^{2,\alpha}$-regularity theory; see \cite[Section 2.3]{LW}. Our main result reads as follows.

\begin{theorem}\label{mainthm}
Let $u\in W^{2,n}(B_1^+)$ be an $L^n$-strong solution to \eqref{maineq}. Suppose Assumptions \ref{assump_F} and \ref{assump_FC1}, to be detailed further, are in force. Then there exist universal constants $C>0$ and $0<\mu\ll 1$ such that 
$$
\left|D^2 u(x)\right| \leq C,
$$
for almost every $x \in B_{1/2}^\mu$.
\end{theorem}

\begin{remark}[General $C^{2,\alpha}$-regular domains]\label{rem_adriatico}\normalfont
    We state Theorem \ref{mainthm} for the upper unit ball $B_1^+$ for simplicity. Indeed, if we prescribe \eqref{maineq} in a $C^{2,\alpha}$-regular domain $\mathcal{U}\subset \mathbb{R}^n$, it would be possible to recover the unit ball setting by flattening the boundary $\partial \mathcal{U}$. In that case, \eqref{maineq} would be driven by an operator $\tilde F:S(n)\times \mathbb{R}^n\times B_1^+\to\mathbb{R}$, which incorporates into $F$ the geometric adjustments stemming from the (local) flattening of $\partial\mathcal{U}$; see \cite[Proposition 2.1]{SS}. Our rationale still applies in that case; it suffices to extend our analysis along the same lines as in \cite{IM2}.
\end{remark}


The remainder of this paper is organised as follows. Section \ref{subsec_ma} details our main assumption, whereas Section \ref{subsec_solv} discusses the notions of solutions used in the paper. We recall former boundary regularity results of second order in Section \ref{subsec_sobolev}. In Section \ref{BMOsec}, we establish BMO-estimates for the solutions to \eqref{maineq} in a uniform neighbourhood of the flat boundary. The proof of Theorem \ref{mainthm} is the subject of Section \ref{mainsec}.

\section{Preliminaries}\label{sec_prelim}

In this section, we collect the foundational material necessary for the developments in the paper. We begin by introducing key definitions and outlining the main assumptions that will be employed throughout our analysis. These elements provide the framework for understanding the problem setting and serve as the basis for the subsequent results. 

\subsection{Main assumptions}\label{subsec_ma}

Denote with $S(n)$ the space of symmetric matrices of order $n$.  For $x\in \mathbb{R}^n$ and $r>0$, we denote with $B_r^+(x)$ the upper hemisphere of the ball of radius $r$ centred at $x\in\mathbb{R}^n$. That is,
\[
    B_r^+(x)\coloneqq\left\lbrace y\in\mathbb{R}^n\,|\,\left\|y-x\right\|<r, \,y_n>0 \right\rbrace.
\]
The flat boundary of $B_r^+(x)$ is denoted with $B_r'(x)$, and given by
\[
    B_r'(x)\coloneqq\left\lbrace y\in\mathbb{R}^n\,|\,\left\|y-x\right\|<r, \,y_n=0 \right\rbrace.
\]
Finally, for $0<\mu\ll1$ and $r>0$, we define the strip ball
$$
    B_{r}^\mu(x)\coloneqq\{y\in B^+_{r}(x)\, | \, y_n < \mu \} \cup B_{r}'.
$$
As usual, we set $B_r^+(0)\eqqcolon B_r^+$, $B_r'(0)\eqqcolon B_r'$, and $B_r^\mu(0)\eqqcolon B_r^\mu$. Our first assumption concerns the operator $F$.

\begin{assumption}[Uniform ellipticity]\label{assump_F}
    We suppose $F:S(n)\to\mathbb{R}$ is a $(\lambda,\Lambda)$-elliptic operator. That is, for every $M,\,N\in S(n)$ we have
    \begin{equation*}\label{eq_vesic}
        \lambda\left\|N\right\|\leq F(M+N)-F(M)\leq\Lambda\left\|N\right\|,
    \end{equation*}
    provided $N\geq 0$. Furthermore, $F(0)=0$.
\end{assumption}
We also require $F$ to satisfy a differentiability condition.

\begin{assumption}[Differentiability of the operator]\label{assump_FC1}
We suppose $F\in C^1(S(n))$. That is, there exists a modulus of continuity $\omega_F:\mathbb{R}_+\to\mathbb{R}_+$ such that 
    \[
        \left|DF(M)-DF(N)\right|\leq\omega_F(|M-N|),
    \]
for every $M,\,N\in S(n)$. 
\end{assumption}

Assumptions \ref{assump_F}--\ref{assump_FC1} completely characterise the class of operators under analysis in the present manuscript. As mentioned, we drop the usual convexity assumption on $F$, imposing a differentiability condition on the operator; see \cite{Savin}. Notice that Assumption \ref{assump_F} naturally generalises to the case of operators with variable coefficients. When dealing with such operators, we require them to have a modulus of continuity with respect to $x\in B_1^+$; this is the content of the next assumption.

\begin{assumption}[Continuity of the operator]\label{assump_continuityop}
    Let $F:S(n)\times B_1^+\to\mathbb{R}$ be a fully nonlinear uniformly elliptic operator. We suppose there exists $C>0$ and $\beta\in (0,1)$ such that 
    \[
        \left|F(M,x)-F(M,y)\right|\leq C|x-y|^\beta\left\|M\right\|,
    \]
    for every $M\in S(n)$ and every $x,y\in B_1^+$.
\end{assumption}

Before we proceed, recall the definition of the Pucci extremal operators. Let $0<\lambda\leq\Lambda$; define $\mathcal{A}_{\lambda,\Lambda}\subset S(n)$ as
\[
	\mathcal{A}_{\lambda,\Lambda}\coloneqq\left\lbrace A\in S(n)\;|\;\lambda|\xi|^2\leq A\xi\cdot\xi\leq\Lambda|\xi|^2\hspace{.1in}\mbox{for every}\hspace{.1in}\xi\in\mathbb{R}^n\right\rbrace.
\]

\begin{definition}[Extremal Pucci operators]\label{def_pucci}
Fix constants $0<\lambda\leq\Lambda$. The extremal Pucci operator $\mathcal{M}^-_{\lambda,\Lambda}:S(n)\to\mathbb{R}$ is given by
\[
	\mathcal{M}^-_{\lambda,\Lambda}(M)\coloneqq\inf_{A\in\mathcal{A}_{\lambda,\Lambda}}\,{\rm Tr}(AM).
\]
We also define $\mathcal{M}^+_{\lambda,\Lambda}(M)\coloneqq-\mathcal{M}^-_{\lambda,\Lambda}(-M)$.
\end{definition}

Notice Assumption \ref{assump_F} can be phrased in terms of $\mathcal{M}^\pm_{\lambda,\Lambda}$. Indeed, $F$ is $(\lambda,\Lambda)$-elliptic if and only if
\begin{equation}\label{eq_francesinha}
    \mathcal{M}^-_{\lambda,\Lambda}(M-N)\leq F(M)-F(N)\leq\mathcal{M}^+_{\lambda,\Lambda}(M-N),
\end{equation}
for every $M,\,N\in S(n)$.

\subsection{Solvability in the $L^p$-strong and $L^p$-viscosity senses}\label{subsec_solv}

We study $L^n$-strong solutions to \eqref{maineq}. However, several arguments in the paper stem from the realm of $L^n$-viscosity solutions. Therefore, we recall the connection between both notions in the context of uniformly elliptic equations \cite{CCKS}. For the sake of completeness, we include the next two definitions. 

\begin{definition}[$L^p$-viscosity solution]\label{def_oxford}
Let $G:S(n)\to\mathbb{R}$ be a $(\lambda,\Lambda)$-uniformly elliptic operator and $f\in L^p(\Omega)$ for some $p>n/2$. We say that $u\in C(\Omega)$ is an $L^p$-viscosity sub-solution to 
    \begin{equation}\label{eq_lorvao}
        G(D^2u,x)=f\hspace{.2in}\mbox{in}\hspace{.2in}\Omega,
    \end{equation}
if, whenever $\phi\in W^{2,p}_{loc}(\Omega)$ is such that $u-\phi$ has a local minimum at $x_0\in\Omega$, we have
\[
	{\rm ess}\limsup_{x\to x_0}\left(G(D^2\phi(x),x)-f(x)\right)\geq 0.
\]
We say that $u\in C(\Omega)$ is an $L^p$-viscosity super-solution to \eqref{eq_lorvao} if, whenever $\phi\in W^{2,p}_{loc}(\Omega)$ is such that $u-\phi$ has a local maximum at $x_0\in\Omega$, we have
\[
	{\rm ess}\liminf_{x\to x_0}\left(G(D^2\phi(x),x)-f(x)\right)\leq 0.
\]
If $u\in C(\Omega)$ is an $L^p$-viscosity sub-solution and an $L^p$-viscosity super-solution to \eqref{eq_lorvao}, we say it is an $L^p$-viscosity solution to \eqref{eq_lorvao}.
\end{definition}

\begin{definition}[$L^p$-strong solution]\label{def_strongsol}
    Let $G:S(n)\to\mathbb{R}$ be a $(\lambda,\Lambda)$-uniformly elliptic operator and $f\in L^p(\Omega)$ for some $p>n/2$. We say that $u\in W^{2,p}(\Omega)$ is an $L^p$-strong solution to \eqref{eq_lorvao} if $G(D^2u(x),x)=f(x)$ for almost every $x\in\Omega$.
\end{definition}

Although Definitions \ref{def_strongsol} and \ref{def_oxford} require $p>n/2$, our arguments impose $p>p_0>n/2$, where $p_0=p_0(\lambda,\Lambda,n)$ is the so-called Escauriaza exponent. That is, the integrability level above which the Aleksandrov--Bakelman--Pucci estimates are available for $L^p$-viscosity solutions to \eqref{eq_lorvao}. Throughout our analysis, we deal with problems whose right-hand side is essentially bounded, ensuring the integrability condition mentioned above is naturally met. Before proceeding, we note that a solution satisfying $\left\|u\right\|_{L^\infty(B_1^+)}\leq 1$ is referred to as a \emph{normalized} solution, both in the $L^p$-strong and $L^p$-viscosity sense.

A fundamental ingredient in our argument is the connection of $L^p$-strong solutions to \eqref{maineq} and $L^p$-viscosity solutions to a uniformly elliptic equation. Suppose $u\in W^{2,p}(B_1^+)$ is an $L^p$-strong solution to \eqref{maineq}. Therefore, 
\begin{equation}\label{eq_mientras}
    F(D^2u(x))=1\hspace{.2in}\mbox{for almost every}\hspace{.2in}x\in B_1^+\cap\Omega.
\end{equation}
However, $\left|D^2u(x)\right|\leq K$ for almost every $x\in B_1^+\setminus\Omega$. Hence, Assumption \ref{assump_F} yields
\[
    \mathcal{M}^-_{\lambda,\Lambda}(D^2u(x))+F(0)\leq F(D^2u(x))\leq\mathcal{M}^+_{\lambda,\Lambda}(D^2u(x))+F(0),
\]
which in turn implies
\begin{equation}\label{eq_todavia}
    \left|F(D^2u(x))\right|\leq C(\lambda,\Lambda,d,K)+F(0)\hspace{.2in}\mbox{for almost every}\hspace{.2in}x\in B_1^+\setminus\Omega.
\end{equation}
By combining \eqref{eq_mientras} and \eqref{eq_todavia}, one gets that 
\begin{equation}\label{eq_deu}
    F(D^2u(x))=f(x)\hspace{.2in}\mbox{for almost every}\hspace{.2in}x\in B_1^+,
\end{equation}
for some $f\in L^\infty(B_1^+)$. Thus an $L^p$-strong solution to \eqref{maineq} is also an $L^p$-strong solution to \eqref{eq_deu}. We establish the following proposition by resorting to \cite[Lemma 2.8]{CCKS}.

\begin{proposition}[$L^p$-viscosity solution]\label{prop_nacional}
    Let $u\in W^{2,p}(B_1^+)$ be an $L^p$-strong solution to \eqref{maineq}, for $p\geq n$. Then there exists $f\in L^p(B_1^+)$ such that $u$ is an $L^p$-viscosity solution to
    \[
        F(D^2u)=f\hspace{.2in}\mbox{in}\hspace{.2in}B_1^+.
    \]
\end{proposition}

In what follows we consider $L^p$-viscosity solutions to
\begin{equation}\label{eq_magicflute}
    \begin{cases}
        \mathcal{M}_{\lambda,\Lambda}^-(D^2v)\leq g\leq\mathcal{M}_{\lambda,\Lambda}^+(D^2v)&\hspace{.2in}\mbox{in}\hspace{.2in}B_1^+\\
        v=0&\hspace{.2in}\mbox{on}\hspace{.2in}\partial B_1^+,
    \end{cases}
\end{equation}
where $g\in L^p(B_1^+)$, for some $n/2<p_0<p$. We resort to an Aleksandrov--Bakelman--Pucci estimate to conclude that solutions to \eqref{eq_magicflute} satisfy
an $L^\infty$-estimate. For completeness, we include it in the sequel in the form of a lemma.

\begin{lemma}\label{lem_abp}
    Let $v\in C(B_1^+)$ be an $L^n$-viscosity solution to \eqref{eq_magicflute}. Suppose $g\in L^p(B_1^+)$, where $n/2<p_0<p$. Then there exists $C>0$ such that 
    \[
        \sup_{x\in B_1^+}|v(x)|\leq C\left\|g\right\|_{L^p(B_1^+)}.
    \]
\end{lemma}
The proof of Lemma \ref{lem_abp} follows from a straightforward application of \cite[Proposition 3.3]{CCKS}. The next section discusses boundary regularity estimates of second order, which play a fundamental role in the proof of Theorem \ref{mainthm}.

\subsection{Boundary regularity estimates of second order}\label{subsec_sobolev}

Our strategy is to produce a BMO-estimate in a uniform neighbourhood of the flat boundary $B_1'$. To that end, our argument builds upon available regularity estimates in $W^{2,p}$ and $C^{2,\alpha}$-spaces. Let $u\in C(B_1^+)$ be an $L^p$-viscosity solution to
\begin{equation}\label{eq_coimbra}
    \begin{cases}
        F(D^2u,x)=f&\hspace{.2in}\mbox{in}\hspace{.2in}B_1^+\\
        u=0&\hspace{.2in}\mbox{on}\hspace{.2in}B_1'.
    \end{cases}
\end{equation}

By imposing different conditions on the ingredients in \eqref{eq_coimbra}, one obtains distinct second-order estimates. We proceed with a proposition.

\begin{proposition}[Sobolev boundary regularity, \protect{\cite[Theorem~4.5]{Win2009}}]\label{prop_honorato}
    Let $v\in C(B_1^+)$ be an $L^p$-viscosity solution to \eqref{eq_coimbra}, with $f\in L^p(B_1^+)$, for $n/2<p_0<p<\infty$. Suppose $F$ satisfies Assumptions \ref{assump_F} and \ref{assump_continuityop}.
    Suppose further that solutions to the homogeneous counterpart of \eqref{eq_coimbra}, driven by the fixed coefficients operator $F(M,x_0)$, are in $C^{1,1}_{\rm loc}\left(B_{1/2}^+(x_0)\right)$, for every $x_0\in B_{1/2}^+$, with estimates. Then $u\in W^{2,p}(B_1^+)$. In addition, there exists $C>0$ such that
    \[
        \left\|u\right\|_{W^{2,p}\left(B_{1/2}^+\right)}\leq C\left(\left\|u\right\|_{L^\infty(B_1^+)}+\left\|f\right\|_{L^p(B_1^+)}\right).
    \]
\end{proposition}
By requiring $F=F(M,x)$ to be differentiable with respect to $M$, we access regularity estimates in $B_1^\mu$ for some universal $0<\mu\ll1$.

\begin{proposition}[$C^{2,\alpha}$-regularity estimates, \protect{\cite[Theorem 1.3]{SS}}]\label{prop_newthm32}
    Let $u\in C(B_1^+)$ be an $L^p$-viscosity solution to \eqref{eq_coimbra}, with $f\equiv 0$. Suppose Assumptions \ref{assump_F}, \ref{assump_FC1}, and \ref{assump_continuityop} hold. Then there exists $\alpha=\alpha(n,\lambda,\Lambda,\beta)$ and $\mu>0$ such that $u\in C^{2,\alpha}(B_1^\mu)$. Moreover, for a universal constant $C>0$, we have
    \[
        \left\|u\right\|_{C^{2,\alpha}(B_1^\mu)}\leq C\left\|u\right\|_{L^\infty(B_1^+)}.
    \]
    Here, $0<\mu\ll1$ depends on the dimension $n$, the ellipticity constants $0<\lambda\leq \Lambda$, $\beta\in(0,1)$ and the modulus of continuity $\omega_F$.
\end{proposition}

We notice Proposition \ref{prop_newthm32} can be phrased in terms of second-order polynomials. In fact, given $x_0\in B_1^\mu$, it is tantamount to the existence of a second-order polynomial $p_{x_0}$ satisfying
\begin{equation}\label{eq_newthm32a}
    \left\|p_{x_0}\right\|_{L^\infty(B_1^+)}\leq C\hspace{.3in}\mbox{and}\hspace{.2in}F(D^2p_{x_0},x_0)=0,
\end{equation}
with
\begin{equation}\label{eq_newthm32b}
    \left|u(x)-p_{x_0}(x)\right|\leq C\left|x-x_0\right|^\alpha,
\end{equation}
for every $x\in B_r(x_0)$, provided $B_r(x_0)\subset B_1^\mu$.

\section{Boundary BMO estimates}\label{BMOsec}

This section considers $L^p$-viscosity solutions to \eqref{eq_coimbra} and establishes BMO-boundary estimates.
We start by examining points universally close to the fixed boundary. That is, we localise the argument in a strip of universal height $\mu>0$ inside $B_1^+$; see Figure \ref{fig_bmo}.

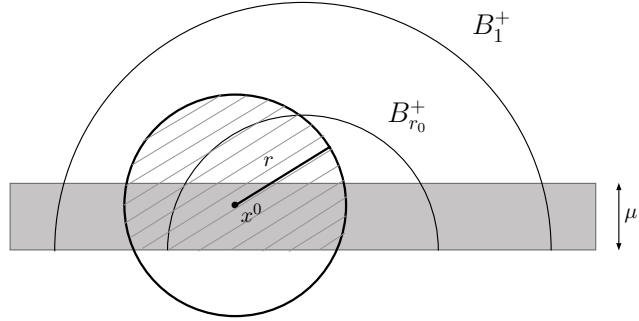
\begin{figure}
\centering
\psscalebox{0.75 0.75}
{
\begin{pspicture}(0,-2.9043834)(12.178064,2.9043834)
\definecolor{colour0}{rgb}{0.7764706,0.76862746,0.76862746}
\definecolor{colour2}{rgb}{0.43137255,0.42352942,0.42352942}
\definecolor{colour1}{rgb}{0.5882353,0.57254905,0.57254905}
\psframe[linecolor=colour2, linewidth=0.02, fillstyle=solid,fillcolor=colour0, dimen=outer](10.4,-0.5043833)(0.0,-1.7043834)
\psarc[linecolor=black, linewidth=0.02, dimen=outer](5.2,-1.7043834){2.4}{0.0}{180.10986}
\psarc[linecolor=black, linewidth=0.02, dimen=outer](5.2,-1.7043834){4.4}{0.0}{180.19302}
\pscircle[linecolor=black, linewidth=0.04, dimen=outer](4.0,-0.90438336){2.0}
\psdots[linecolor=black, dotsize=0.12](4.0,-0.90438336)
\psline[linecolor=black, linewidth=0.02, arrowsize=0.05291667cm 2.0,arrowlength=1.4,arrowinset=0.0]{<->}(10.8,-0.5043833)(10.8,-1.7043834)
\psline[linecolor=black, linewidth=0.04](4.0,-0.90438336)(5.6903224,0.14077795)
\psline[linecolor=colour1, linewidth=0.01](2.356774,0.22981021)(3.743871,1.0485198)
\psline[linecolor=colour1, linewidth=0.01](2.1696775,-0.18954463)(4.2019353,1.0420682)
\psline[linecolor=colour1, linewidth=0.01](2.06,-0.55728656)(4.595484,0.9711005)
\psline[linecolor=colour1, linewidth=0.01](2.027742,-0.8992221)(4.8858066,0.86142313)
\psline[linecolor=colour1, linewidth=0.01](2.0341935,-1.2282543)(5.1503224,0.7259392)
\psline[linecolor=colour1, linewidth=0.01](2.1567743,-1.466964)(5.376129,0.53884244)
\psline[linecolor=colour1, linewidth=0.01](2.356774,-1.6669639)(5.576129,0.33884245)
\psline[linecolor=colour1, linewidth=0.01](2.7946408,-1.7081684)(5.7189074,0.09617586)
\psline[linecolor=colour1, linewidth=0.01](3.3430278,-1.71462)(5.8414884,-0.19414672)
\psline[linecolor=colour1, linewidth=0.01](3.943028,-1.6823618)(5.9253592,-0.45866284)
\psline[linecolor=colour1, linewidth=0.01](4.478512,-1.71462)(5.9834237,-0.7489854)
\psline[linecolor=colour1, linewidth=0.01](4.9494796,-1.7210716)(5.964069,-1.0909209)
\psline[linecolor=colour1, linewidth=0.01](5.4978666,-1.6888136)(5.944714,-1.3941467)
\rput[bl](10.9,-1.2){$\mu$}
\rput[bl](8.2,2.0){{\Large $B_1^+$}}
\rput[bl](6.7,0.4){{\Large $B_{r_0}^+$}}
\rput[bl](4.1,-1.2){$x^0$}
\rput[bl](4.5,-0.2){$r$}
\end{pspicture}
}
\caption{Our analysis produce a radius $r_0>0$ and a height $\mu>0$ such that $L^n$-strong solutions to \eqref{eq_coimbra} satisfy a universal estimate in spaces $W^{2,{\rm BMO}}$. I.e., the Hessian of $u$ is uniformly bounded in the BMO-norm.}\label{fig_bmo}
\end{figure}
 
We proceed with the statement of the main result in this section.

\begin{theorem}\label{BMOthm} Let u be an $L^n$-viscosity solution of \eqref{eq_coimbra}. Suppose Assumptions \ref{assump_F} and \ref{assump_FC1} are in force. Suppose further that $f \in L^\infty(B_1^+)$. There exist universal constants $C>0$ and $0<\mu,r_0\ll1$ such that, for each $0<r \leq r_0$ and $x^0 \in B_{r}^\mu$, one finds a second-order polynomial $p_{r,x^0}$ satisfying 
$$
F(D^2p_{r,x^0},x^0)=0,
$$
$$
|D^2 p_{2r,x^0}-D^2 p_{r,x^0}| \leq C\max\left\lbrace 1,\|u\|_{L^\infty(B_1^+)}\right\rbrace,
$$ 
and
\begin{equation}\label{BMOest}
\sup\limits_{B_{r}(x^0)\, \cap B^+_{1}}\left|u(x)-p_{r,x^0}(x) \right| \leq C\max\left\lbrace1,\|u\|_{L^\infty(B_1^+)}\right\rbrace\, r^2. 
\end{equation}  
\end{theorem}
Theorem \ref{BMOthm} follows from an iterative procedure stemming from Proposition \ref{prop_newthm32}. It relies on the $C^{2,\alpha}$ regularity estimates at points near the flat boundary available in that proposition. To state and prove the building blocks of Theorem \ref{BMOthm}, we consider points $x^0 \in B_{1/2}^\mu$, for $0<\mu\ll1$ yet to be chosen.

\begin{proposition}\label{step0}
Let $u\in C(B_1^+)$ be a normalized $L^p$-viscosity solution to \eqref{eq_coimbra}. Suppose Assumptions \ref{assump_F}, \ref{assump_FC1}, and \ref{assump_continuityop} are in force. Then there exist universal constants $C>0$ and $\varepsilon_0>0$ such that, for $0<\varepsilon \leq \varepsilon_0$ one finds $\delta>0$ for which, if
\begin{equation}\label{smallcond}
\|f\|_{L ^\infty(B_1^+)} \leq \,\delta \quad \mbox{and} \quad \sup_{t>0} \omega(t) \leq \,\delta, 
\end{equation} 
there exists a second-order polynomial $p$ satisfying 
$$
\|p\|_{L^{\infty}(B_{1}^+)} \leq C \quad \mbox{and} \quad F(D^2p,x^0)=0,
$$
with
$$
\sup_{x \in B_{\varepsilon}(x^0) \,\cap \, B_{1}^+}|u(x)-p(x)| \leq \varepsilon^{2}.
$$
\end{proposition}
\begin{proof}
We argue by contradiction and split the proof into three steps for an easy presentation. 

\medskip

\noindent{\bf Step 1 - }Suppose the proposition statement is false. Hence, there exist $\varepsilon_0>0$ and sequences $(F_k)_{k\in\mathbb{N}}$, $(\omega_k)_{k\in\mathbb{N}}$, $(u_k)_{k\in\mathbb{N}}$, and $(f_k)_{k\in\mathbb{N}}$ such that $u_k$ is a normalized $L^p$-viscosity solution to \eqref{eq_coimbra} driven by $F_k$, with
$$
\max\left\lbrace\sup\limits_{t>0}\omega_k(t),\|f_k\|_{L^\infty(B_1^+)}\right\rbrace \leq \frac 1k.
$$
but
\begin{equation}\label{contradiction}
\sup_{\substack{\|p\|_{L^{\infty}} \leq 2C \\ F_k(D^2 p)=0}} \left[ \sup_{x \in B_{\varepsilon_0}(x^0)\, \cap B_{3/4}^+}|u_k(x)- p(x)| \right] > \varepsilon_0^2
\end{equation}
for every $k\in\mathbb{N}$, where $C>0$ is given as in \eqref{eq_newthm32a}-\eqref{eq_newthm32b}.

\vspace{.2in}

\noindent{\bf Step 2 - }Applying uniform global estimates to $u_k$, as stated in \cite[Proposition 2.6]{SS}, we extract a subsequence still denoted with $(u_k)_{k\in\mathbb{N}}$ which converges uniformly in $B^+_{3/4}$ to some $u_\infty\in C(B_{3/4}^+)$. Standard stability results in the theory of $L^p$-viscosity solutions ensure
\begin{equation}\label{unifeq2}
\begin{cases}
F_\infty(D^2 u_\infty)= 0 & \hspace{.2in}\mbox{in}\hspace{.2in} B_{3/4}^+ \\
u_\infty=0 & \hspace{.2in}\mbox{on} \hspace{.2in}B_{3/4}', 
\end{cases}
\end{equation}
where $F_\infty$ is a constant coefficients operator, satisfying Assumptions \ref{assump_F} and \ref{assump_FC1}. Hence, Proposition \ref{prop_newthm32} implies that, for
$$
0 < \varepsilon_0 \leq \left(\frac{1}{3C}\right)^{\frac{1}{\alpha}},
$$
we have
$$
\sup\limits_{B_{\varepsilon_0}(x^0)} |u_\infty(x)-p_{x^0}(x)| \leq \frac{1}{3}\varepsilon_0^{2},
$$
for some quadratic polynomial $p_{x^0}$ satisfying 
\[
    \|p_{x^0}\|_{L^\infty(B_{3/4}^+)} \leq C
\]
and
\[
F_\infty(D^2p_{x^0})=0\hspace{.2in}\mbox{in} \hspace{.2in}B_{1/2}^+.
\]

\vspace{.2in}

\noindent{\bf Step 3 - }Recall that Assumptions \ref{assump_F}, \ref{assump_FC1} and \ref{assump_continuityop} are in force for $F_k$, for every $n\in\mathbb{N}$. We also have $F_n(0,x)=0$ for every $n\in\mathbb{N}$. Hence, there exists a sequence of real numbers $(a_k)_{k\in\mathbb{N}}$ with $a_k \to 0$ such that 
\[
F_k(D^2p_n+a_kI,x^0)=0.
\]      
From this and the fact that $u_k \to u_\infty$ uniformly in $B_{3/4}^+$, we have for $k  \gg 1$, that
\begin{equation}
\begin{split}
\sup\limits_{B_{\varepsilon_0}(x^0)\, \cap B_{3/4}^+}|u_k-p_k| & \leq \sup\limits_{B_{\varepsilon_0}(x^0) \cap B_{3/4}^+}|u_\infty-p_{x^0}| \\
 &\quad + \sup\limits_{B_{3/4}^+}|u_k-u_\infty| + \sup\limits_{B_{3/4}^+}|p_k(x) - p(x)| \\
 & \leq  \frac{1}{3}\varepsilon_0^{2}+\frac{1}{3}\varepsilon_0^{2}+\frac{1}{3}\varepsilon_0^{2}.
\end{split}
\end{equation}  
The former inequality yields a contradiction in light of \eqref{contradiction} and completes the proof.
\end{proof}
 
\begin{proposition}\label{propk}
Let $u\in C(B_1^+)$ be a normalized $L^p$-viscosity solution to \eqref{eq_coimbra}. Suppose Assumptions \ref{assump_F}, \ref{assump_FC1}, and \ref{assump_continuityop} are in force.  There exist universal constants $C>0$, $\delta_0>0$ and $\rho>0$, such that if 
\begin{equation}\label{smallcond2}
\left\|f\right\|_{L ^\infty(B_1^+)} + \sup\limits_{0<t\le 1}\omega(t) \leq \delta_0   
\end{equation}  
then for each $k \in \mathbb{N}$ there is a second-order polynomial $p_k$, satisfying 
\[
\left|D^2p_k-D^2p_{k-1}\right| \leq C 
\]
and
\[
F(D^2p_k,x^0)=0,
\]
with
\begin{equation}\label{finaleq}
\sup_{x \in B_{\rho^k}(x^0)\, \cap B_{1}^+}|u(x)- p_k(x)| \leq \rho^{2k}.
\end{equation}
\end{proposition}
\begin{proof}
We argue by induction in $k\in\mathbb{N}$. As before, we split the proof into two steps.

\vspace{.2in}

\noindent{\bf Step 1 - }Set $p_0=p_{-1}=0$; the case $k=0$ follows because $u$ is a normalized solution. Suppose the thesis of the proposition has been verified for $k=j$. We proceed by verifying it holds for $k=j+1$. Indeed, we define  
$$
u_j(x)\coloneqq\frac{u(\rho^jx+x^0)-p_l(\rho^jx+x^0)}{\rho^{2j}},
$$
for $x\in B_1^+$.
The induction hypothesis ensures $u_j$ is normalised. Notice it solves
$$
F_j(D^2u_j,x)=f_l \hspace{.2in} \mbox{in}\hspace{.2in} B_1^+, 
$$
for $F_j(M,x)=F(M+D^2 p_j,\rho^jx+x^0)$, $f_j\coloneqq f(\rho^jx+x^0)$ and $\omega_j(t)\coloneqq\omega(\rho^j  t)$. The new modulus of continuity $\omega_k$ still falls within the scope of \ref{assump_FC1}. Also, \ref{assump_F} and \ref{assump_continuityop} remain in force at the $j$-th level. Finally, we observe that
\[
\left\|f_j\right\|_{L^\infty(B_1^+)} \leq \left\|f\right\|_{L^\infty(B_1^+)} \leq \delta_0
\]
and
\[
\sup\limits_{t>0}\omega_j(t) \leq \sup\limits_{t>0}\omega(\rho\,t) \leq \delta_0.
\]

\vspace{.2in}

\noindent{\bf Step 2 - }Take $\delta_0=\delta$ depending on $\varepsilon_0=\rho$, as in Proposition \ref{step0}. We conclude there exists a polynomial 
\[
    p(x)=\frac 12 x^TAx+ {\rm b}\cdot x+ \kappa,
\]
with universally bounded coefficients $A \in \mathbb{R}^{n\times n}$, ${\rm b} \in \mathbb{R}^n$, and $\kappa \in \mathbb{R}$, such that $F_j(A,0)=0$ and
$$
\sup_{B_{\rho}(x_j)\, \cap B^+_{3/4}}|u_j(x)- p(x)| \leq \rho^{2}
$$
for $x_j\coloneqq\rho^{-j} x^0$. It implies that
$$
\sup_{B_{\rho^{j+1}}(x^0)\, \cap B_{3/4}^+}|u(x)- p_{j+1}(x)| \leq \rho^{2(j+1)},
$$
for  
\[
    p_{j+1}(x)\coloneqq\overline p(x)+p_j(x)
\]
where 
\[
\overline p(x)\coloneqq\frac 12 x^TAx+\rho^j{\rm b} x+ \rho^{2j}\kappa.
\] 
In addition, we note that
$$
F(D^2 p_{j+1},x^0) = F(A + D^2p_j,x^0)=F_j(A,0)=0;
$$
a further application of Proposition \ref{step0} ensures
$$
|D^2p_{j+1}-D^2p_{j}| = |A| \leq C,
$$ 
which completes the induction argument and ends the proof.
\end{proof}
Now, we are in a position to detail the proof of Theorem \ref{BMOthm}. 

\begin{proof}[Proof of Theorem \ref{BMOthm}]
We split the proof into three steps.

\vspace{.2in}

\noindent{\bf Step 1 - }We start with an application of Proposition \ref{propk}. To that end, we proceed with a scaling argument. For parameters $\tau$ and $\kappa$, to be chosen universally, define
$$
\overline{u}(x)\coloneqq \kappa \cdot u(\tau x) \quad \mbox{in } \; B_1^+.
$$ 
Notice $\overline u$ solves
$$
\overline F(D^2 \overline u, x )=\overline f(x) \hspace{.2in} \mbox{in}\hspace{.2in} \quad B_{1}^+
$$
where 
\[
\overline F(M,x)\coloneqq\kappa \tau^2\cdot F([\kappa\tau^2]^{-1} \cdot M, \tau x)
\]
and
\[
\overline f(x)=k\tau^2f(\tau x).
\]
Note that $\overline F$ satisfies Assumptions \ref{assump_F}, \ref{assump_FC1}, and \ref{assump_continuityop}. Set
$$
\kappa \leq \min\left\{1,\|u\|_{L^{\infty}(B_1^+)}^{-1}\right\} \quad \mbox{and} \quad \tau \leq \sqrt{\min\left\{\delta\|f\|_{L^{\infty}(B_1^+)}^{-1},\omega^{-1}(\delta)\right\}}, 
$$
to ensure that $\overline{u}$ is a normalized solution and \eqref{smallcond2} is in force. 

\vspace{.2in}

\noindent{\bf Step 2 - }In this setting, we resort to Proposition \ref{propk}, obtaining \eqref{finaleq} centred at $\tau^{-1}x^0$. That is, one obtains a sequence of second-order polynomials $(p_k)_{k\in\mathbb{N}}$, such that
$$
|D^2 p_k-D^2 p_{k-1}| \leq C \quad \mbox{and} \quad \overline F(D^2 p_k,\tau^{-1}x^0)=0.
$$
for every $k\in\mathbb{N}$. Moreover, for the universal radius $\rho>0$, we ensure
\begin{equation*}
\begin{split}
    \kappa \sup\limits_{x \in B_{\tau\rho^k}(x^0)\, \cap B_{1}^+}|u(x)- \tilde p_k(x)|&\leq \kappa \sup\limits_{x \in B_{\tau\rho^k}(x^0)\, \cap B_{\tau}^+}|u(x)- \tilde p_k(x)| \\
        &\leq \sup\limits_{x \in B_{\rho^k}(\tau^{-1} x^0)\, \cap B_{1}^+}|\overline u(x)- p_k(x)| \\
        &\leq \rho^{2k},
\end{split}
\end{equation*}
where $\tilde p_k(x)\coloneqq\kappa^{-1} p_k(\tau^{-1} x)$. We also note that
$$
F(D^2 \tilde p_k,x^0)=\tau^{-2}\kappa^{-1}\overline F(D^2 p_k,\tau^{-1}x^0)=\tau^{-2}\kappa^{-1}\overline F(D^2p_k,\tau^{-1}x^0)=0.
$$

\vspace{.2in}

\noindent{\bf Step 3 - }Finally, for $0<r<\tau\rho$, we select $p_{r,x^0}\coloneqq\tilde p_k$ where $k=k(r)$ is the positive integer such that 
$$
\tau\rho^{k+1} < r \leq \tau\rho^k.
$$
Hence,
\begin{equation}\nonumber
\begin{array}{rcl}
\sup\limits_{x \in B_{r}(x^0)\, \cap B_1^+}|u(x)- p_{r,x^0}(x)| & \leq & \sup\limits_{x \in B_{\tau \rho^k}(x^0)\, \cap B_1^+}|u(x)- \tilde p_k(x)| \\[0.6cm]
 & \leq & \kappa^{-1}\rho^{2k}\\[0.3cm]
 & \leq & (\tau^{-2}\rho^{-2}\kappa^{-1})\tau^2\rho^{2(k+1)}\\[0.3cm]
 & \leq & C\kappa^{-1}r^{2}.  
\end{array}
\end{equation}
In addition, for any 
\[
0<r<\frac{1}{2}\tau\rho,
\]
we also obtain
\begin{equation}
    \begin{split}
        \sup\limits_{x \in B_r(x^0)\, \cap B_1^+}\left|p_{2r,x^0}-p_{r,x^0}\right| & \leq \sup\limits_{x \in B_{2r}(x^0)\, \cap B_1^+}\left|u-p_{2r,x^0}\right|\\
            &\quad +\sup\limits_{x \in B_r(x^0)\, \cap B_1^+}\left|u-p_{r,x^0}\right|\\
            & \leq \overline Cr^2 
\end{split} 
\end{equation}
for some universal $\overline C>0$. Since $\|p_{r,x^0}\|_{L^\infty(B_1^+)}$ is universally bounded, one gets
\begin{equation}
|D^2 p_{2r,x^0}-D^2 p_{r,x^0}| \leq C,
\end{equation}
for some $C>0$. By noticing that such a constant is universal, one completes the proof.
\end{proof}

\medskip
 
\section{Proof of Theorem \ref{mainthm}} \label{mainsec}

In this section, we detail the proof of our main result. We follow ideas in \cite{FS14}, adapting them carefully for the boundary scenario. To simplify our arguments, we rewrite Theorem \ref{BMOthm} in terms of polynomials approximating the Hessian of solutions; see \cite[Lemma 2.2]{FS14}.

\begin{corollary}\label{honorato}
Let $u$ be an $L^n$-strong solution to \eqref{maineq}. Let $0<\mu\ll1$ be the universal constant from Theorem \ref{BMOthm}. Fix $x_0\in B_{1/2}^{\mu/2}$ and suppose
\[
	u(x_0)=|Du(x_0)|=0.
\]
Suppose further Assumptions \ref{assump_F} and \ref{assump_FC1} are in force. There exist universal constants $C>0$ and $r_0>0$, such that, for each $0<r\leq r_0$, one finds $p_{r,x_0} \in S(n)$, satisfying $F(p_{r,x_0})=1$, 
\begin{equation}\label{estA}
|p_{2r,x_0}-p_{r,x_0}| \leq C\max\{1,\|u\|_{L^\infty(B_1^+)}\},
\end{equation}
and
\begin{equation}\label{est1}
\sup_{B^+_r(x_0)}\left| u - \frac{1}{2}\langle p_{r,x_0}\, (y-x_0), y-x_0\rangle \right| \leq C r^2.   
\end{equation}
In particular, there holds
\begin{equation}\label{est2}
\sup_{B^+_r(x_0)} |u| \leq (C+|p_{r,x_0}|)r^2.   
\end{equation}
\end{corollary}
\begin{proof} 
Proposition \ref{prop_nacional} ensures $u$ is entitled to the conclusion of Theorem \ref{BMOthm}. Consequently, we consider $\mu$ and $r_0$ as defined in that theorem. For each $0<r \leq r_0$ and $x^0 \in B_{r_0}^\mu$, there exists a second-order polynomial $p_{r,x^0}$ such that
\[
F(D^2 p_{r,x^0})=0,
\]
\[
|D^2p_{2r,x^0}-D^2 p_{r,x^0}| \leq C\max\{1,\|u\|_{L^\infty(B_1^+)}\},
\]
and
\begin{equation}\label{BMOest'}
\sup\limits_{B_{r}(x^0)\, \cap B^+_{1}}\left|u(x)-p_{r,x^0}(x) \right| \leq C\max\{1,\|u\|_{L^\infty(B_1^+)}\}\, r^2. 
\end{equation}

Define
$$
v(x)=\frac{u(rx-x_0)-p_{r,x^0}(rx-x_0)}{r^2},
$$
for $x \in B_1^+$. Notice that $v$ is a normalized $L^p$-viscosity solution to
$$
G(D^2 v)=g  \hspace{.2in}\mbox{in }\hspace{.2in} B_1^+,
$$
where $G(M)\coloneqq F(M+P_{r,x^0})$, and $g \in L^\infty(B_1^+)$ satisfies $\left\|g\right\|_{L^\infty(B_1^+)} \leq \min\left\lbrace1,K\right\rbrace$. 
Since $u(x_0)=|Du(x_0)|=0$, standard regularity estimates imply
\begin{equation}\label{eurocopa}
\frac{|Dp_{r,x^0}(x_0)|r+|p_{r,x^0}(x_0)|}{r^2} \leq |Dv(0)|+|v(0)| \leq C,
\end{equation}
for some universal $C>0$. 

In addition, Assumption \ref{assump_F} builds upon the equality $F(D^2 p_{r,x^0})=0$ to produce a universal constant $\delta \in \mathbb{R}$ such that, for $\tilde p_{r,x^0}\coloneqq D^2p_{r,x^0}+\delta I_n$, we have
$$
F(\tilde p_{r,x^0})=1.
$$
Finally, combining \eqref{BMOest'}, \eqref{eurocopa}, and the triangle inequality, we conclude 
$$
\sup\limits_{B_{r}(x^0)\, \cap B^+_{1}}\left|u(x)-\frac{1}{2}\langle \tilde p_{r,x^0}(x-x_0),x-x_0\rangle \right| \leq \left(C(1+\|u\|_{L^\infty(B_1^+)})+\delta\right)\, r^2.
$$
Hence, \eqref{est1} holds. Also, estimate \eqref{est2} follows from \eqref{est1}. Renaming $\tilde p_{r,x_0}$ as $p_{r,x_0}$ one finishes the proof.
\end{proof}

Hereafter, we suppose $x_0 \in \partial\Omega$ and $u(x_0)=|Du(x_0)|=0$. For $r>0$ to be chosen universally small, we define 
$$
H_{r,x_0}^+\coloneqq \frac{\left(B_r^+(x_0)\setminus\Omega\right)-x_0}{r}.
$$ 
In other words, $H_r$ is the rescaled region where the Hessian is bounded a priori. 

\begin{remark}[Geometry of $H_{r,x_0}^+$]\label{rem_geomh}\normalfont
For the reader's convenience, we notice two properties of $H_{r,x_0}^+$, which are helpful in the upcoming arguments. For each $\beta, \rho \in (0,1)$, we first claim that if $x \in  H^+_{\beta,x_0} \cap B_\rho^+(x_0)$ then $\frac{x}{\rho} \in  H^+_{\beta\rho,x_0}$. In fact, taking $x$ as above, we denote $z=(x/\rho)$ and note that
$$
(\beta \rho) z  = \beta x \notin \Omega, \quad \mbox{with} \quad |z| \leq 1;
$$
hence, $z \in H^+_{\beta\rho,x_0}$. We also claim that, if $y \in  H^+_{\beta,x_0}$, then $\rho y \in  H^+_{\frac{\beta}{\rho},x_0}\cap B_\rho^+(x_0)$. Indeed, $\beta y \in B_\beta^+(x_0)$ and $\beta y \notin \Omega$; therefore,
$$
\frac{\beta}{\rho}(\rho y) \in B^+_\beta(x_0) \subset B^+_{\beta/\rho}(x_0) \quad \mbox{and} \quad \frac{\beta}{\rho}(\rho y) \notin \Omega.
$$ 
\end{remark}
Next, assuming $P_r$ is large, we prove that $H_r^+$ has a universal $L^n$-decay.
\begin{proposition}\label{mondego}
Let $u$ be an $L^n$-strong solution to \eqref{maineq}. Let $0<\mu\ll1$ be the universal constant from Theorem \ref{BMOthm}. Fix $x_0\in B_{1/2}^{\mu/2}$ and suppose
\[
	u(x_0)=|Du(x_0)|=0.
\]
 Suppose Assumptions \ref{assump_F} and \ref{assump_FC1} are in force. There exists a universal constant $M$ such that if 
$$
|p_{r,x_0}| \geq M,
$$
then 
$$
|H^+_{r/2,x_0}\cap B_\mu^+(x_0)| \leq \frac{|H^+_{r,x_0}\cap B_\mu^+(x_0)|}{2^n},
$$
for every $0<r\ll1$.
\end{proposition}
\begin{proof}
Set
\[
u_{r,x_0}(x)\coloneqq\dfrac{u(x_0+rx)}{r^2}
\]
and let $v_{r,x_0}$ be the solution of
\begin{equation}\nonumber
\begin{cases}
F(p_{r,x_0}+D^2v_{r,x_0})-1=0 & \hspace{.2in}\mbox{in} \hspace{.2in}B_1^+ \\
v_{r,x_0}=u_{r,x_0}-\langle p_{r,x_0}(x-x_0), x-x_0 \rangle & \hspace{.2in} \mbox{on} \hspace{.2in}\partial B_1^+.
\end{cases}
\end{equation} 
From Corollary \ref{honorato}, we infer that $F(p_{r,x_0})=1$, and that $v_{r,x_0}$ is universally bounded on $\partial B_1^+$. Also, $F(p_{r,x_0}+M)$ is of class $C^1$, with the same modulus of continuity as $F(M)$. Hence, an application of \cite[Theorem 1.3]{SS} guarantees that 
\begin{equation}\label{CMUC}
\|D^2 v_r\|_{C^{0,\alpha}(B^+_\mu)} \leq \|D^2 v_r\|_{C^{0,\alpha}(B_{1/2}^\mu)} \leq C,
\end{equation}
for some universal $\alpha\in(0,1)$. Now, we define
$$
w_{r,x_0}\coloneqq u_{r,x_0}-\frac{1}{2}\langle p_{r,x_0}(x-x_0), x-x_0 \rangle-v_{r,x_0} \quad \mbox{in }\; B_1^+.
$$
Since $f_{r,x_0}\coloneqq F(D^2 u_{r,x_0}) \in L^{\infty}(B_1^+)$ and $f_{r,x_0} \equiv 1$ outside $H_{r,x_0}^+$, we observe that 
$$
F(D^2u_{r,x_0})-F(p_{r,x_0}+D^2v_r)=(f_{r,x_0}-1) \,\chi_{H_{r,x_0}^+}.
$$
Assumption \ref{assump_F} implies
\begin{equation*}
\begin{cases}
\mathcal{M}^-_{\lambda,\Lambda}(D^2w_{r,x_0}) \leq (f_{r,x_0} -1) \cdot\chi_{H_{r,x_0}^+} \leq \mathcal{M}^+_{\lambda,\Lambda}(D^2w_{r,x_0}) &\hspace{.2in}\mbox{in}\hspace{.2in}B_1^+ \\
w_{r,x_0} = 0&  \hspace{.2in}\mbox{on} \hspace{.2in}\partial B_1^+. 
\end{cases}
\end{equation*} 
Because $f_{r,x_0}$ is universally bounded, we apply Lemma \ref{lem_abp} to conclude 
\begin{equation}\label{academica}
\sup_{B_1^+}|w_{r,x_0}| \leq C\|\chi_{H_{r,x_0}^+}\|_{L^n(B_1^+)}=C|H_{r,x_0}^+|^{1/n}.
\end{equation}
In parallel, we notice that $w_{r,x_0}$ solves
$$
G(D^2 w_{r,x_0}, x)=(f_{r,x_0}-1)\chi_{H_{r,x_0}^+} \quad \mbox{in }\; B_{\mu}^+,
$$
where $G(M,x)\coloneqq F(p_{r,x_0}+D^2v_{r,x_0}(x)+M)-1$. From estimate \eqref{CMUC}, we observe that $G$ is a fully nonlinear operator with H\"older coefficients in $B_\mu^+$. To apply Proposition \ref{prop_honorato}, we must examine the solutions' regularity to $G(M,x_0)=0$.

Indeed, 
\[
    G(M,x_0)=F(p_{r,x_0}+D^2v_{r,x_0}(x_0)+M)-1.
\]
It is clear that $G$ is $(\lambda,\Lambda)$-elliptic. Furthermore, 
\[
    DG(M,x_0)=DF(p_{r,x_0}+D^2v_{r,x_0}(x_0)+M);
\]
hence, the modulus of continuity of $DG$ coincides with the one prescribed in Assumption \ref{assump_FC1}. We conclude that solutions to $G(D^2w,x_0)=0$ in $B_\mu^+$, with $w=0$ on $B_\mu'$, are locally $C^{2,\alpha}$-regular. Proposition \ref{prop_honorato} builds upon the former discussion to yield
\[
	\begin{split}
		\int_{B_{\mu/2}^+(x_0)}\left|D^2 w_{r,x_0}\right|^{2 n}& \leq C\left(\left\|w_{r,x_0}\right\|_{L^{\infty}(B_{\mu/2}^+(x_0))}+\|\chi_{H_{r,x_0}^+}\|_{L^{2 n}(B_{\mu/2}^+(x_0))}\right)^{2 n} \\
			&\leq C|H^+_{r,x_0} \cap B_{\mu/2}^+(x_0)|.
	\end{split}
\]
Where we have used \eqref{academica}. 

To conclude the proof, observe that \eqref{maineq} implies
$$
\displaystyle\int_{H^+_{r,x_0} \cap B_{\mu/2}^+(x_0)}|D^2u_{r,x_0}|^{2n}\, dx \leq K^{2n}|H^+_{r,x_0} \cap B_{\mu/2}^+(x_0)| \leq K^{2n}|H_{r,x_0}^+|.
$$
Therefore,
\begin{equation}\nonumber
\begin{split}
|H^+_{r,x_0} \cap B_{\mu/2}^+(x_0)||P_r|^{2n}& =  \displaystyle\int_{H^+_{r,x_0} \cap B_{\mu/2}^+(x_0)}|p_{r,x_0}|^{2n}\, dx \\
& \leq  C\displaystyle\int_{H^+_r \cap B_{\mu/2}^+(x_0)}|D^2u_{r,x_0}|^{2n}\, dx\\
&\quad + \displaystyle\int_{H^+_r \cap B_{\mu/2}^+(x_0)}|D^2v_{r,x_0}|^{2n} + |D^2 w_{r,x_0}|^{2n}\, dx \\
& \leq  K^{2n}|H^+_{r,x_0}\cap B_\mu^+(x_0)| +C|H^+_{r,x_0} \cap B_{\mu/2}^+(x_0)|.
\end{split}
\end{equation}
Hence, we suppose $|P_r|$ is universally large so that
\begin{equation}\label{pingodoce}
	\begin{split}
		|H^+_{r,x_0} \cap B_{\mu/2}^+(x_0)||p_{r,x_0}|^{2n} &\leq \overline{C}|H^+_{r,x_0}\cap B_\mu^+(x_0)| \\
			&\leq \frac{1}{4^n}|p_{r,x_0}|^{2n}|H^+_{r,x_0}\cap B_{\mu}^+(x_0)|,
	\end{split}
\end{equation}
for some unversal $\overline{C}>0$. Remark \ref{rem_geomh} yields
$$
\frac{(H^+_{r/2,x_0} \cap B_{\mu}^+(x_0))}{2} \subset (H^+_{r,x_0} \cap B_{\mu/2}^+(x_0)),
$$
where
\[
\frac{(H^+_{r/2,x_0} \cap B_{\mu}^+(x_0))}{2}\coloneqq\{y \in B_1^+\, | \, 2y \in H^+_{r/2,x_0} \cap B_{\mu}^+(x_0) \}.
\]
The former inclusion holds because
$$
ry=\frac{r}{2}(2y) \in B_{r/2}^+(x_0)\setminus \Omega \subset B_{r}^+(x_0)\setminus \Omega,
$$
for each 
\[
y\in \frac{(H^+_{r/2,x_0} \cap B_{\mu}^+(x_0))}{2}.
\]
Therefore,
$$
|H^+_{r/2,x_0} \cap B_{\mu}^+(x_0)| \leq 2^n|H^+_{r,x_0} \cap B_{\mu/2}^+(x_0)|.
$$

Finally, from \eqref{pingodoce}, we conclude that 
$$
|H^+_{r/2,x_0} \cap B_{\mu}^+(x_0)| \leq \frac{1}{2^n} |H^+_{r,x_0}\cap B_{\mu}^+(x_0)|,
$$
which gives us the desired estimate.
\end{proof}

In the sequel, we put forward the proof of Theorem \ref{mainthm}.

\begin{proof}[Proof of Theorem \ref{mainthm}]
Because $u\in W^{2,n}(B_1^+)$, it is twice differentiable almost everywhere. We suppose $x_0$ is a Lebesgue point for $u$, and $u(x_0)=|Du(x_0)|=0$. For $M$ as in Proposition \ref{mondego}, the following alternative arises: either
\[
    \liminf_{k\to \infty}|p_{2^{-k},x_0}|\leq 3M
\]
or 
\begin{equation}\label{eq_fridaynight}
    \liminf_{k\to\infty}|p_{2^{-k},x_0}|\geq 3M.
\end{equation}
In the former case, an application of Corollary \ref{honorato} yields
\begin{equation}
\begin{array}{rcl}
|D^2 u(x_0)| & = & \displaystyle \lim_{k\to \infty} \intav{B^+_{2^{-k}}(x_0)} |D^2 u(y)|\, dy \\
& = & \displaystyle \lim_{k\to \infty} \intav{B^+_{2^{-k}}(x_0)} |D^2 u(y)-p_{2^{-k},x_0}|\, dy + |p_{2^{-k},x_0}| \\
 & \leq & C+3M,
\end{array}
\end{equation}
and concludes the proof. It remains to consider \eqref{eq_fridaynight}.

\vspace{.2in}

\noindent{\bf Step 2 - }If \eqref{eq_fridaynight} is in force, we define $k_0\in\mathbb{N}$ as 
\[
    k_0\coloneqq\inf\left\lbrace k\geq 2,\hspace{.1in}\mbox{such that}\hspace{.1in}|p_{2^{-j},x_0}|\geq 2M \hspace{.1in}\mbox{for every}\hspace{.1in}j\geq k\right\rbrace.
\]
We claim that $k_0$ is finite; indeed, because of \eqref{eq_fridaynight}, there must be $k_1\in\mathbb{N}$ such that $|p_{2^{-k_1},x_0}|>2M$. By the very definition of $k_0$, it follows that $k_0\leq k_1$. It also follows from the definition of $k_0$ that $|p_{2^{-k_0-1},x_0}|\leq 2M$. Now, we resort to Corollary \ref{honorato} for the first time in the proof. In fact, \eqref{estA} yields
\[
    \left|p_{2^{-k_0},x_0}\right|\leq \left|p_{2^{-k_0},x_0}-p_{2^{-k_0-1},x_0}\right|+\left|p_{2^{-k_0-1},x_0}\right|\leq C+2M.
\]

\vspace{.2in}

\noindent{\bf Step 3 - }We continue by defining the function $\overline u_0$ as
\[
    \overline u_0(x)\coloneqq 4^{k_0}u(2^{-k}x+x_0)-\frac{1}{2}\left\langle p_{2^{-k_0},x_0}(x-x_0),x-x_0\right\rangle.
\]
Notice $\overline u_0$ solves
\begin{equation}\label{eq_G}
    F(D^2\overline u_0+p_{2^{k_0},x_0})-1=\left(f_{2^{-k_0},x_0}-1\right)\chi_{H_{2^{-k_0},x_0}^+}\hspace{.2in}\mbox{in}\hspace{.2in}B_1^+
\end{equation}
We want to apply a boundary variant of Caffarelli's regularity estimates. First notice that $\tilde F(M)\coloneqq F(M+p_{2^{k_0},x_0})-1$ satisfies Assumptions \ref{assump_F} and \ref{assump_FC1}, with the same constants. 

Now, we examine $\tilde f\coloneqq (f_{2^{-k_0},x_0}-1)\chi_{H_{2^{-k_0}}^+}$. Because $|p_{2^{-k},x_0}|\geq 2M$ for every $k\geq k_0$, Proposition \ref{mondego} implies 
\[
    \left|H_{2^{-k_0-j},x_0}^+\cap B_\mu^+(x_0)\right|\leq \frac{1}{2^{jn}}\left|H_{2^{-k_0},x_0}^+\cap B_\mu^+(x_0)\right|\leq\frac{1}{2^{jn}}\left|B_1^+\right|,
\]
for every $j\geq k_0$. As a consequence, for 
\[
	r\leq\min\left(\mu,2^{-k_0}\right),
\]
we get
\begin{equation}\label{eq_diospiro}
\intav{B_r^+}\left|(f_{2^{-k_0},x_0}-1)\chi_{H_{2^{-k_0},x_0}^+}\right|^n{\rm d}x\leq C\intav{B_r^+}\chi_{H_{2^{-k_0},x_0}^+\cap B_\mu^+(x_0)}{\rm d}x\leq Cr^n,
\end{equation}
for some universal constant $C>0$.

Because of Proposition \ref{prop_honorato}, solutions to $\tilde F=0$ have $C^{2,\alpha}$-regularity estimates at $x=0$. The geometric decay in \eqref{eq_diospiro} ensures $\overline u_0$ is of class $C^{2,\alpha}$ at $x=0$ (see, for instance, \cite[Theorem 2.7]{LW}). Therefore,
\[
    \left|D^2u(x_0)\right|\leq \left|D^2\overline u_0(0)\right|+\left|p_{2^{-k_0},x_0}\right|\leq C,
\]
for some universal constant $C>0$, which completes the proof.
\end{proof}

\medskip

\noindent{\bf Acknowledgments:} DJA is partially supported by CNPq-Brazil, grant 311138/2019-5, and Para\'iba State Research Foundation (FAPESQ), grant 2019/0014. EP is partially supported by the Centre for Mathematics of the University of Coimbra (CMUC, funded by the Portuguese Government through FCT/MCTES, DOI 10.54499/UIDB/00324/2020).

\bibliographystyle{amsalpha}
\bibliography{biblio}

\end{document}